\documentclass[12pt]{article}

\usepackage{amsmath}
\usepackage{amsthm}
\usepackage{amsxtra}
\usepackage{bbm}

\usepackage{amsfonts,amssymb}
\usepackage{latexsym}

\newtheorem{Th}{Theorem}
\newtheorem{Prop}{Proposition}

\theoremstyle{definition}
\newtheorem{Def}{Definition}

\newtheorem{Rem}{Remark}

\newcommand{\aut}{\mathrm{Aut}}

\newcommand{\supp}{\mathrm{supp}}

\newcommand{\sym}{\mathrm{Sym}}
\newcommand{\rist}{\mathrm{rist}}

\newcommand{\dd}{\mathrm{d}}

\newcommand{\eg}{\text{e.g.\,}}

\newcommand{\etc}{\text{etc.}}

\numberwithin{equation}{section}

\begin{document}
\title{On irreducibility of Koopman representations of Higman-Thompson groups.}
\author{ {\bf Artem Dudko}  \\
                    Stony Brook University, Stony Brook, NY, USA  \\
          artem.dudko@stonybrook.edu
         }

\date{}

\maketitle
\begin{abstract} We introduce a notion of measure contracting actions and show that Koopman representations corresponding to ergodic measure contracting actions are irreducible. As a corollary we obtain that Koopman representations associated to canonical actions of Higman-Thompson groups are irreducible. We also show that the actions of weakly branch groups on the boundaries of rooted trees are measure contracting. This gives a new point of view on irreducibility of the corresponding Koopman representations.
\end{abstract}
\section{Introduction}
One of the most natural representations that one can associate to a measure class preserving action of a group $G$ on a measure space $(X,\mu)$, where $\mu$ is a probabilty measure, is the Koopman representation $\kappa$ of $G$ in $L^2(X,\mu)$ defined by:$$(\kappa(g)f)(x)=\sqrt{\frac{\dd\mu(g^{-1}(x))}{\dd\mu(x)}}f(g^{-1}x).$$ This representation is important due to the fact that the spectral properties of $\kappa$ reflect the measure-theoretic and dynamical properties of the action such as ergodicity and weak-mixing.

 It is known that for an ergodic action operators $\kappa(g)$ together with operators of multiplication by functions from $L^\infty(X,\mu)$ generate the algebra of all bounded operators on $L^2(X,\mu)$. 
 A natural question is whether the operators $\kappa(g),g\in G$ generate the algebra of all bounded operator by themselves, that is whether $\kappa$ is irreducible. Below are several examples of group actions with quasi-invariant measures for which the Koopman representation is known to be irreducible:
\begin{itemize}
\item{} actions of free non-commutative  groups on their boundaries (\cite{FTP83} and \cite{FTS94});
\item{} actions of lattices of Lie-groups (or algebraic-groups) on their Poisson-Furstenberg boundaries (\cite{CS91} and \cite{BC02});
\item{} action of the fundamental group of a compact negatively curved manifold on its boundary endowed with the Paterson-Sullivan measure class (\cite{BM11});
\item{} natural actions of Thompson's groups $F$ and $T$ on the unit segment (\cite{Garn12});
\item{} action of the group of compactly supported contactomorphisms of a contact manifold (\cite{Garn14});
\item{} actions of weakly branch groups on the boundaries of the corresponding rooted trees (\cite{DuGr-Koop}).
\end{itemize}
However, the general question in what cases the Koopman representation is irreducible remains open.

In the present paper we introduce a notion of a \emph{measure contracting action} and show that Koopman representations corresponding to ergodic measure contracting actions are irreducible. Using the above we show that Koopman representations corresponding to natural actions of Higman-Thompson groups are irreducible and reconsider the Koopman representations of weakly branch groups presented in \cite{DuGr-Koop}. We also notice that the measure contracting property applies to other interesting group actions. For example, \L{}ukasz Garncarek pointed to the author that the results of the present paper can be used to prove irreducibility of the Koopman representation of the group of inner automorphims of a foliation.

Apparently, the most famous group from the family of Higman-Thompson groups is  the Thompson group $F_{2,1}$ consisting of all piecewise linear continuous transformations of the unit interval with singularities at the points $\{\tfrac{p}{2^q} : p,q\in\mathbb N\}$ and slopes in $\{2^q : q\in \mathbb Z\}$. This group satisfies a number of unusual properties and disproves several important conjectures in group theory. The group $F_{2,1}$ is infinite but finitely presented, is not elementary amenable, has exponential growth, and does not contain a subgroup isomorphic to the free group of rank 2. An important open question is whether the group $F_{2,1}$ is amenable. Further discussion of historical importance of Higman-Thompson groups and their various algebraic properties can be found  in \cite{Brown:1987}, \cite{CFP:1996}, and \cite{BS14}. Each of the groups $G_{n,r}$ and $F_{n,r}$ acts canonically on $[0,r]$. Denote by $\lambda_r$ the Lebesgue probability measure on $[0,r]$. We show the following:
\begin{Th}\label{ThHT} Let $G$ be a group from the Higman-Thompson families $\{F_{n,r}\}$, $ \{G_{n,r}\}$. Then the Koopman representation of $G$ corresponding to the canonical action of $G$ on $([0,r],\lambda_r)$ is irreducible.\end{Th}

After the first version of the present paper was submitted to the arXiv \L{}ukasz Garncarek called the author's attention to \cite{Garn12} where irreducibility of a more general class of Koopman type representations (with Radon-Nikodym derivative twisted by a cocycle) of the Thompson groups $F=F_{1,2}$ and $T$ is proven. In fact,  Garncarek's methods readily adapt to Koopman representations of Higman-Thompson groups. Thus, Theorem \ref{ThHT} can be attributed to Garncarek. However, the proof obtained here is different.

In \cite{DM14 Thompson} the author of the present paper jointly with Medynets showed that Higman-Thompson groups have only discrete set of finite type factor-representations using the notion of a \emph{compressible action}. We notice that the notion of measure contracting actions we introduce in the present paper is loosely related to the notion of compressible actions.

The second class of representations we consider is Koopman representations of weakly branch groups. A group acting on a rooted tree $T$ is called weakly branch if it acts transitively on each level of the tree and for every vertex $v$ of $T$ it has a nontrivial element $g$ supported on the subtree $T_v$ emerging from $v$ (see \eg \cite{BGS03} and \cite{Grig11}). Weakly branch groups posses  interesting  and  often  unusual  properties. The class of weakly branch groups contains groups  of  intermediate  growth, amenable  but  not  elementary  amenable  groups,  groups  with  finite  commutator  width  \etc. Weakly branch groups  also  play  important  role  in  studies  in holomorphic  dynamics (see \cite{Nekr}) and  in the  theory  of  fractals  (see \cite{GNS15}).

For a $d$-regular rooted tree $T$ its boundary $\partial T$ can be identified with a space of sequences $\{x_j\}_{j\in\mathbb N}$ where $x_j\in\{1,\ldots,d\}$. For a collection of positive real numbers $p=\{p_1,\ldots,p_d\}$ with $p_1+\ldots+p_d=1$ let $\mu_p$ be the corresponding Bernoulli measure on $\partial T$. In \cite{DuGr-Koop} the authors showed the following:
 \begin{Th}\label{ThBranchIrred} Let  $G$ be a subexponentially bounded weakly branch group acting on a regular rooted tree and $p$ be as above such that $p_i$ are pairwise distinct. Then the Koopman representation associated to the action of $G$ on  $(\partial T,\mu_p)$ is irreducible.
\end{Th}
\noindent Here subexponentially bounded group means a group consisting of subexponentially bounded (in the sense similar to polynomial boundedness of Sidki \cite{Sid00}) automorphisms of $T$. In the present paper using the results of \cite{DuGr-Koop} we show that actions of subexponentially bounded weakly branch groups on $(\partial T,\mu_p)$ (with $p_i$ pairwise distinct) are measure contracting. This gives a different view on the proof of Theorem \ref{ThBranchIrred} presented in \cite{DuGr-Koop}.

\subsection*{Acknowledgements.} The author acknowledges \L{}ukazs Garncarek for important comments and essential references. The author is grateful to Rostislav Grigorchuk for valuable remarks and useful suggestions.

\section{Measure contracting actions.}
\begin{Def}\label{DefMeasContr} Let $G$ act on a probability space $(X,\mu)$ with a quasi-invariant measure $\mu$. We will call this action \emph{measure contracting} if for every measurable subset $A\subset X$ and any $M,\epsilon>0$ there exists $g\in G$ such that
\begin{itemize}
\item[$1)$] $\mu(\supp(g)\setminus A)<\epsilon$;
\item[$2)$] $\mu(\{x\in A:\sqrt{\tfrac{\dd\mu(g(x))}{\dd\mu(x)}}<M^{-1}\})>\mu(A)-\epsilon$.
\end{itemize}
\end{Def}
\noindent Here $\supp(g)=\{x\in X:gx\neq x\}$. For a measure space $(X,\mu)$ denote by $\mathcal L(X,\mu)$ the von Neumann algebra generated by operators of multiplication by functions from $L^\infty(X,\mu)$ on $L^2(X,\mu)$. Let $\mathcal B(X,\mu)$ be the algebra of all bounded linear operators on $L^2(X,\mu)$. For a set of operators $\mathcal S\subset \mathcal B(X,\mu)$ denote by $$\mathcal S'=\{A\in \mathcal B(X,\mu):AB=BA\;\;\text{for all}\;\;B\in \mathcal S\}$$ the commutant of $\mathcal S$.
The following result is folklore.
\begin{Th}\label{ThGen} Let group $G$ act ergodically by measure class preserving transformations on a standard Borel space $(X,\mu)$. Then the von Neumann algebra $\widetilde{\mathcal M}_\kappa$ generated by operators from $\mathcal M_\kappa$ and  $\mathcal L^\infty(X,\mu)$ coincides with $\mathcal B(X,\mu)$.
\end{Th}
\begin{proof} By von Neumann bicommutant theorem (see \eg \cite{bratelli_robinson:1982}, Theorem 2.4.11) it is sufficient to show that the commutant $\widetilde{\mathcal M}_\kappa'$ consists only from scalar operators. Let $A\in \widetilde{\mathcal M}_\kappa'$. Since $\mathcal L(X,\mu)$ is a maximal abelian subalgebra of $\mathcal B(X,\mu)$ (see \eg \cite{Dix69}, Lemma 8.5.1) we obtain that $A\in \mathcal L(X,\mu)$. That is, $A$ is the operator of multiplication by a function $m\in L^\infty(X,\mu)$. Since $A$ commute with $\kappa(g)$ for all $g\in G$ the function $m$ is $G$-invariant ($m(gx)=m(x)$ for all $g\in G$ for almost all $x\in X$). By ergodicity, $m$ is constant almost everywhere. Therefore, operator $A$ is scalar. This finishes the proof.
\end{proof}
The main result of this section is the following:
\begin{Th}\label{ThKoopContr} For any ergodic measure contracting action of a group $G$ on a probability space $(X,\mu)$ the associated Koopman representation $\kappa$ of $G$ is irreducible.
\end{Th}
\begin{proof} First, for every measurable subset $A\subset X$ fix a sequence of elements $g_m^A$ such that
\begin{itemize}
\item[$1)$] $\mu(\supp(g_m^A)\setminus A)<\tfrac{1}{m}$,
\item[$2)$] $\mu(\{x\in A:\sqrt{\tfrac{\dd\mu(g_m^A(x))}{\dd\mu(x)}}<\tfrac{1}{m}\})>\mu(A)-\tfrac{1}{m}$
\end{itemize} for every $m\in\mathbb N$.
 For a subset $B\subset X$ denote by $P^B$ the orhtogonal projection onto the subspace
 $$\mathcal H^B=\{\eta\in L^2(X,\mu):\supp(\eta)\subset X\setminus B\}.$$
Let us show that for every measurable subset $A\subset X$ one has \begin{equation}\label{EqProj}w-\lim\limits_{m\to\infty}\pi(g_m^A)=P^A,
\end{equation} where $w-\lim$ stands for the limit in the weak operator topology.

Fix a measurable subset $A\subset X$. Introduce the sets
$$A_m=\{x\in A:\sqrt{\tfrac{\dd\mu(g_m^A(x))}{\dd\mu(x)}}<\tfrac{1}{m}\},\;\;B_m=\supp(g_m^A)\setminus A_m.$$
Notice that $$\mu(A\setminus A_m)<\tfrac{1}{m}\;\; 
\text{and}\;\; \mu(B_m)<\tfrac{2}{m}.$$
To prove \eqref{EqProj} it is sufficient to show that for any $\eta_1,\eta_2\in L^2(X,\mu)$ one has
$$(\pi(g_m^A)\eta_1,\eta_2)\rightarrow (P^A\eta_1,\eta_2)$$ when $m\to\infty$. Since the subspace of essentially bounded functions $L^2(X,\mu)\cap L^\infty(X,\mu)$ is dense in $L^2(X,\mu)$ we can assume that $\eta_1$ and $\eta_2$ are essentially bounded. Let $M_i=\|\eta_i\|_\infty,i=1,2$. We have:
\begin{align*} |(\pi((g_m^A)^{-1})\eta_1,\eta_2)-(P^A\eta_1,\eta_2)|=\bigg|\int\limits_X \sqrt{\tfrac{\dd\mu(g_m^A(x))}{\dd\mu(x)}}\eta_1(g_m^Ax)\overline{\eta_2(x)}\dd x-\\
\int\limits_{X\setminus A}\eta_1(x)\overline{\eta_2(x)}\dd x\bigg|\leqslant \tfrac{1}{m}\bigg|\int\limits_{A_m}\eta_1(g_m^Ax)\overline{\eta_2(x)}\dd x\bigg|+\\
\bigg|\int\limits_{B_m}\sqrt{\tfrac{\dd\mu(g_m^A(x))}{\dd\mu(x)}}
\eta_1(g_m^Ax)\overline{\eta_2(x)}\dd x\bigg|+\bigg|\int\limits_{B_m}
\eta_1(x)\overline{\eta_2(x)}\dd x\bigg|\leqslant \\ \tfrac{1}{m}M_1M_2+|(\pi((g_m^A)^{-1})\eta_1,P^{X\setminus B_m}\eta_2)|+\tfrac{2}{\sqrt m}M_1M_2\rightarrow 0
\end{align*} when $m\to\infty$, since $\|\pi((g_m^A)^{-1})\eta_1\|=\|\eta_1\|$ (in the norm on $L^2(X,\mu)$) and $P^{X\setminus B_m}\eta_2\to 0$ when $m\to\infty$. Observe that for every $g\in G$ one has $(\pi(g^{-1})\eta_1,\eta_2)=\overline{(\pi(g)\eta_2,\eta_1)}$. This finishes the proof of \eqref{EqProj}. In particular, we obtain that $P^A\in\mathcal M_\kappa$ for every measurable subset $A\subset X$.

Observe that every function from $L^\infty(X,\mu)$ can be approximated arbitrarily well in $L^2$-norm by finite linear combinations of characteristic functions of open sets. This implies that for every $m\in L^\infty(X,\mu)$ the operator of multiplication by $m$
$$\mathcal H\to \mathcal H,\;\;f\to mf$$ can be approximated arbitrary well in the strong operator topology by finite linear combinations of projections $P_A\in\mathcal M_{\kappa}$, and thus belongs to the von Neumann algebra $\mathcal M_{\kappa}$ generated by operators $\kappa(g),g\in G$. It follows that $\mathcal M_{\kappa}$ contains the algebra $\mathcal L(X,\mu)$. Using Theorem \ref{ThGen} we obtain that $\mathcal M_\kappa$ coincides with $\mathcal B(X,\mu)$. This finishes the proof of Theorem \ref{ThKoopContr}.\end{proof}
\begin{Rem} Garncarek pointed to the author that the proof of Theorem \ref{ThKoopContr} works to show irreducibility of a more general class of Koopman type representations. Namely, representations of the form:
$$(\kappa_\gamma(g)f)(x)=\gamma(g^{-1},x)\sqrt{\frac{\dd\mu(g^{-1}(x))}{\dd\mu(x)}}f(g^{-1}x),$$ where $\gamma:G\times X\to \{z\in\mathbb C:|z|=1\}$ is any cocycle.
\end{Rem}
\section{Higman-Thompson groups}
Let us briefly recall the definition of Higman-Thompson groups. For details we refer the reader to \cite{Brown:1987}, \cite{CFP:1996}, and \cite{BS14}.
\begin{Def} Fix two positive integers $n$ and $r$. \\
$(1)$ The group $F_{n,r}$ consists of all orientation preserving piecewise linear homeomorphisms $h$ of $[0,r]$ such that all singularities of $h$ are in $\mathbb Z[1/n]=\{\tfrac{p}{n^k}:p,k\in\mathbb N\}$ and
the derivative of $h$ at any non-singular point is $n^k$ for some $k\in \mathbb Z$.\\
$(2)$ The group $G_{n,r}$ is the group of all right continuous piecewise linear bijections $h$ of $[0,r)$ with finitely many discontinuities  and singularities, all in $\mathbb Z[1/n]$, such that the derivative of $h$ at any non-singular point is $n^k$ for some $k\in \mathbb Z$ and $h$ maps $\mathbb Z[1/n]\cap [0,r)$   to itself.
\end{Def}
\begin{Prop}\label{PropThompsonErg} Let $G$ be a group from the Higman-Thompson families $\{F_{n,r}\}$, $ \{G_{n,r}\}$. Then the canonical action of $G$ on $([0,r],\lambda_r)$ is ergodic.
\end{Prop}
\begin{proof} Since $F_{n,r}<G_{n,r}$ without loss of generality we can assume that $G=F_{n,r}$ for some $n,r$. Let $A$ be a $G$-invariant measurable subset of $[0,r]$ such that $0<\lambda_r(A)<1$. Denote by $\Lambda$ the set of segments $I\subset (0,r)$ of the form
$I=\big[\tfrac{(n-1)p}{n^m},\tfrac{(n-1)(p+1)}{n^m}\big],$ where $p,m\in \mathbb N$. Corollary A5.6 from \cite{BS14} implies that for any $I_1,I_2\in\Lambda$ there exists $g\in G$ which maps $I_1$ onto $I_2$. Replacing, if necessary, $g$ on $I_1$ by the linear orientation preserving map sending $I_1$ onto $I_2$ we may assume that $g'(x)$ is constant on $I_1$. $G$-invariance of $A$ implies that $$\frac{\lambda_r(A\cap I_1)}{\lambda_r(I_1)}=\frac{\lambda_r(A\cap I_2)}{\lambda_r(I_2)}.$$ It follows that
$\tfrac{\lambda_r(A\cap I)}{\lambda_r(I)}$ does not depend on $I\in\Lambda$. Since every measurable subset $B\subset [0,r]$ can be approximated by measure arbitrarily well by finite unions of segments from $\Lambda$ we obtain that the ratio $\frac{\lambda_r(A\cap B)}{\lambda_r(B)}$ is the same for all measurable subset $B\subset [0,r]$ with $\lambda_r(B)>0$. Taking $B=A$ and $B=[0,r]$ we arrive at a contradiction which finishes the proof.
\end{proof}
\begin{Prop}\label{PropThompsonCommens}  Let $G$ be a group from the Higman-Thompson families $\{F_{n,r}\}$, $ \{G_{n,r}\}$. Then the canonical action of $G$ on $([0,r],\lambda_r)$ is measure contracting.
\end{Prop}
\begin{proof} Since $F_{n,r}<G_{n,r}$ it is sufficient to consider the case $G=F_{n,r}$. Introduce a sequence $g_m$ of elements of $G$ as follows:
\begin{equation}g_m(x)=\left\{\begin{array}{ll}x,&\text{if}\;\;0\leqslant x<\tfrac{r}{n^{2m}},\\
n^mx-\tfrac{r}{n^{m}}+\tfrac{r}{n^{2m}},&\text{if}\;\;\tfrac{r}{n^{2m}}\leqslant x<\tfrac{r}{n^{m}},\\
r-\tfrac{r}{n^m}+\tfrac{x}{n^m},&\text{if}\;\;\tfrac{r}{n^{m}}\leqslant x\leqslant r.
\end{array}\right.
\end{equation}
Observe that $\lambda_r(\{x\in [0,r]:\sqrt{\tfrac{\dd\lambda_r(g_m(x))}{\dd\lambda_r(x)}}\geqslant n^{-m}\})=n^{-m}$.

For any segment $I\subset [0,r]$ of the form \begin{equation}\label{EqI}I=\big[r\tfrac{p}{n^m},r\tfrac{p+1}{n^m}\big],\end{equation} where $m\in\mathbb N$ and $p\in \mathbb Z_+$, let $J_I$ be the unique increasing affine map sending $[0,r]$ onto $I$. Introduce elements $g_m^I\in F_{n,r}$ by
$$g_m^I(x)=\left\{\begin{array}{ll}J_Ig_mJ_I^{-1}(x),&\text{if}\;\;x\in I,\\x,&\text{otherwise}.\end{array}\right.$$
We will call a  set $A\subset [0,1]$ admissible if it is a finite union of segments of the form \eqref{EqI}. For an admissible set $A\subset [0,r]$ and any sufficiently large $m\in\mathbb N$ fix a partition $A=I_1\cup I_2\cup\ldots\cup I_k$, where each of $I_j$ is of the form \eqref{EqI} and $I_j$ intersect $I_l$ for $l\neq j$ by at most one point, and set
$g_m^A=g_m^{I_1}g_m^{I_2}\cdots g_m^{I_k}$. 
Clearly, one has $$ \supp(g_m^A)\subset A\;\;\text{and}\;\;\lambda_r(\{x\in A:\sqrt{\tfrac{\dd\lambda_r(g_m(x))}{\dd\lambda_r(x)}}\geqslant n^{-m}\})=n^{-m}\lambda_r(A).$$ Since any measurable subset of $[0,r]$ can be approximated by measure arbitrarily well by admissible sets the latter implies that the action of $G$ on $[0,r]$ is measure contracting.
\end{proof}
\noindent As a Corollary of Propositions \ref{PropThompsonErg} and \ref{PropThompsonCommens}, and Theorem \ref{ThKoopContr} we obtain Theorem \ref{ThHT}.
\section{Weakly branch groups.}
First, let us give a brief introduction to weakly branch groups. See \eg \cite{BGS03} and \cite{Grig11} for details.

A rooted tree is a tree $T$ with vertex set divided into levels $V_n$, $n\in\mathbb Z_+$, such that $V_0$ consists of one vertex $v_0$ (called the root of $T$), the edges are only between consecutive levels, and each vertex from $V_n$, $n\geqslant 1$ (we consider infinite trees), is connected by an edge to exactly one vertex from $V_{n-1}$ (and several vertexes from $V_{n+1}$). A rooted tree is called spherically homogeneous if each vertex from $V_n$ connected to the same number $d_n$ of vertexes from $V_{n+1}$. $T$ is called $d-$regular, if $d_n=d$ is the same for all levels. For a vertex $v$ of a rooted tree $T$ denote by $T_v$ the subtree emerging from $v$. The automorphism group $\aut(T)$ consists of all automorphisms of $T$ (as a graph) preserving the root.


\begin{Def}\label{DefBr} Let $T$ be a spherically homogeneous tree and $G<\aut(T)$. Rigid stabilizer of a vertex $v$  is the subgroup $\rist_v(G)$ consisting of element $g$ acting trivially outside of $T_v$.
 Rigid stabilizer of level $n$ is
$$\rist_n(G)=\prod\limits_{v\in V_n}\rist_v(G).$$ $G$ is called \emph{branch} if it is transitive on each level and $\rist_n(G)$ is a subgroup of finite index in $G$ for all $n$. $G$ is called \emph{weakly branch} if it is transitive on each level $V_n$ of $T$ and $\rist_v(G)$ is nontrivial for each $v$.
\end{Def}

The boundary $\partial T$ of a $d$-regular rooted tree is homeomorphic to a space of infinite sequences $\{1,2,\ldots,d\}^\mathbb{N}$ and hence is homeomorphic to a Cantor set. For a $d$-tuple $p=(p_1,\ldots,p_d)$ of such that \begin{equation}\label{EqP}p_i>0\;\;\text{for all}\;\;i\;\;\text{and}\;\;p_1+\ldots +p_d=1\end{equation} define a measure $\nu_p$ on $\{1,2,\ldots, d\}$ by $$\nu_p(\{1\})=p_1,\nu_p(\{2\})=p_2,\ldots,\nu_p(\{d\})=p_d.$$ Let $\mu_p=\nu_p^\mathbb N$ be the corresponding Bernoulli measure on $\partial T$.  For each level $V_n$ of a $d$-regular rooted tree an automorphism $g$ of $T$ can be presented in the form
$$g=\sigma\cdot(g_1,\ldots,g_{d^n}),$$ where $\sigma\in\sym(V_n)$ is a permutation of the vertexes from $V_n$ and $g_i$ are the restrictions of $g$ on the subtrees emerging from the vertexes of $V_n$. 
\begin{Def}\label{DefSubexp} We will call an element $g\in\aut(T)$ subexponentially bounded (in the sense similar to polynomial boundedness of Sidki \cite{Sid00}) if the numbers $k_n(g)$ of restrictions $g_i$ to the vertexes of level $n$ not equal to identity satisfy:
$$\lim\limits_{n\to\infty}k_n(g)\gamma^n=0\;\;\text{for any}\;\;0<\gamma<1.$$ A group $G<\aut(T)$ is subexponentially bounded if each $g\in G$ is subexponentially bounded.
\end{Def}
In \cite{DuGr-Koop}, Proposition 2 the authors showed that for any subexponentially bounded  automorphism $g$ and any $p$ as in \eqref{EqP} the measure $\mu_p$ is quasi-invariant with respect to the action of $g$.
For a subexponentially bounded weakly branch group  $G$ acting on a $d$-regular rooted tree and $p$ be as in \eqref{EqP} denote by $\kappa_p$ the Koopman representation corresponding to the action of
 $G$ on $(\partial T,\mu_p)$. Let $\xi_A\in L^2(\partial T,\mu_p)$ stand for the characteristic function of a measurable subset $A\subset \partial T$.
 From \cite{DuGr-Koop} (Corollary 3 and Lemma 4) we deduce the following:
 \begin{Prop}\label{PropGn} Let  $G$ be a subexponentially bounded weakly branch group acting on a $d$-regular rooted tree and $p$ be as in \eqref{EqP} with pairwise distinct $p_i$. Then for every clopen set $A\subset\partial T$ there exists a sequence of elements $g_n\in G$ with $\supp(g_n)\subset A$ such that:
 \begin{equation}\label{EqGn}\lim\limits_{n\to\infty} (\kappa_p(g_n)\xi_A,\xi_A)=0.\end{equation}
 \end{Prop}
 \noindent
 Using results of \cite{DuGr-Koop} we show:
 \begin{Prop} Let  $G$ be a subexponentially bounded weakly branch group acting on a $d$-regular rooted tree and $p_1,p_2,\ldots,p_d\in (0,1)$ such that $p_i$ are pairwise distinct and $\sum\limits_{i=1}^dp_i=1$. Then the action of $G$ on $(\partial T,\mu_p)$ is measure contracting.
 \end{Prop}
 \begin{proof} Let $A$ be a measurable subset of $(\partial T,\mu_p)$. Since clopen sets approximate all measurable subsets of $(\partial T,\mu_p)$ by measure arbitrarily well to show that Definition \ref{DefMeasContr} is satisfied we may assume that $A$ is clopen. Let $g_n$ be a sequence of elements from Proposition \ref{PropGn}. Clearly, for large enough $n$ condition $1)$ from Definition \ref{DefMeasContr} is satisfied for $g_n$. Assume that for some $\epsilon, M>0$ for all $n$ condition $2)$ does not hold. Set $$B_n=\{x\in A:\sqrt{\tfrac{\dd\mu_p(g_n(x))}{\dd\mu_p(x)}}> M^{-1}\}.$$ Then $\mu_p(B_n)\geqslant \epsilon$ for all $n$  and we have:
 $$(\kappa_p(g_n)\xi_A,\xi_A)=(\kappa_p(g_n^{-1})\xi_A,\xi_A)\geqslant \int\limits_{B_n} \sqrt{\tfrac{\dd\mu_p(g_n(x))}{\dd\mu_p(x)}}\dd\mu_p(x)\geqslant M^{-1}\epsilon.$$ This contradicts to \eqref{EqGn}. It follows that for large enough $n$ condition $2)$ from Definition \ref{DefMeasContr} is also satisfied for $g_n$. This finishes the proof.
 \end{proof}


\end{document}